\documentclass[11pt, reqno]{amsart}
\usepackage{a4wide}
\usepackage{amssymb}
\usepackage{amsmath}
\usepackage{amsthm}
\usepackage{tikz}
\usepackage{amscd}

\usepackage[curve]{xypic}
\usepackage{hyperref}
\usepackage[all]{xy}
\usepackage{color}
\theoremstyle{plain}
\usepackage[margin=1in]{geometry}

\newcommand{\id}{\operatorname{id}}

\newcommand{\integers}[1]{\operatorname{\mathcal{O}_{{\it #1}}}}
\newcommand{\ideal}[1]{\operatorname{\mathfrak{P}_{{\it #1}}}}
\newcommand{\ind}{\operatorname{ind}}

\newcommand{\res}{\operatorname{res}}
\newcommand{\ho}{\operatorname{Hom}}
\newcommand{\mir}{\operatorname{Mir}}

\newcommand{\g}[2]{\operatorname{GL}_{#1}(#2)}
\newcommand{\diag}{\operatorname{diag}}
\newcommand{\tr}{\operatorname{tr}}

\newcommand{\Rad}{\operatorname{Rad}}

\newtheorem{theorem}{Theorem}[section]
\newtheorem{corollary}[theorem]{Corollary}
\newtheorem{lemma}[theorem]{Lemma}
\newtheorem{remark}[theorem]{Remark}
\newtheorem{proposition}[theorem]{Proposition}

\newtheorem{definition}[theorem]{Definition}

\author{\large{ Santosh Nadimpalli }}
\date{\today}

\begin{document}
\title{Typical representations for level zero Bernstein components of $\g{n}{F}$}
\maketitle
\begin{abstract}
  Let $F$ be a non-discrete non-Archimedean locally compact field. In
  this article for a level zero Bernstein component $s$, we classify
  those irreducible smooth representations of $\g{n}{\integers{F}}$
  (called typical representations) whose appearance in a smooth
  irreducible representation $\pi$ of $\g{n}{F}$ implies that the
  cuspidal support of $\pi$ is $s$.  These results extend, for level
  zero representations, the results of Henniart and Pa\v{s}k\={u}nas on
  cuspidal representations. The results are independent of the
  characteristic of the base field.
\end{abstract}

\section{Introduction}

Let $F$ be a non-discrete non-Archimedean locally compact field.
 The isomorphism classes of irreducible smooth complex
representations of $\g{n}{F}$, denoted by  $\mathcal{A}_n$, can be decomposed
as a disjoint union 
$$\mathcal{A}_n=\prod_{s\in \mathcal{B}_n}\mathcal{A}_n(s),$$
where $\mathcal{A}_n(s)$ is defined in terms of parabolic induction
and the parameter $s$ is called the Bernstein component or inertial
support. In the context of the local Langlands correspondence, the
parameter $s$ determines the isomorphism class of the restriction to
the inertia subgroup $I_F$ of the Weil--Deligne representation associated by
the classical local Langlands correspondence.

The Local class field theory gives a natural isomorphism between $I_F$
and $\mathcal{O}_F^{\times}$, the group of units of the ring of
integers of $F$. It is natural to ask for a relation between the
representations of $I_F$ which can be extended to a Weil--Deligne
representation and the representations of the maximal compact
subgroup $\g{n}{\integers{F}}$. One natural way would be to understand the
cuspidal support of a smooth irreducible representation from its
restriction to $\g{n}{\integers{F}}$. Indeed in several arithmetic
applications (see \cite{Henniart-gl_2}, \cite{geo_breuil_mez}) it is
desired to construct irreducible smooth representations $\tau_s$ of
the maximal compact subgroup $\g{n}{\integers{F}}$ such that for any
irreducible smooth representation $\pi$ of $\g{n}{F}$,
$$\ho_{ \g{n}{\integers{F}}}(\tau_s, \pi)\neq 0 \Rightarrow \pi
\in \mathcal{A}_n(s).$$
Such a representation $\tau_s$ is called a {\bf typical
  representation} for $s$. In this article we completely classify
typical representations $\tau_s$ for all level-zero Bernstein
components (see section 3) of $\g{n}{F}$.

The existence of typical representation, for any $s$, follows from the
theory of types developed by Bushnell and Kutzko. For all
$s\in \mathcal{B}_n$, Bushnell and Kutzko constructed pairs
$(J_s, \lambda_s)$ such that for any irreducible smooth representation
$\pi$ of $\g{n}{F}$,
$$\ho_{ J_s}(\lambda_s, \pi)\neq 0 \Leftrightarrow \pi
\in \mathcal{A}_n(s).$$
It follows from Frobenius reciprocity that any irreducible
subrepresentation of 
\begin{equation}\label{typical}
\ind_{J_s}^{\g{n}{\integers{F}}}\lambda_s
\end{equation}
is a typical representation for $s$. In general the representation
\eqref{typical} is not irreducible and it is not known if there are
any other typical representations which are not subrepresentations of
\eqref{typical}. 

For $n=2$, Henniart (see \cite{Henniart-gl_2}) classified typical
representations for all inertial classes.  Later Pa\v{s}k\={u}nas (see
\cite{Paskunas-uniqueness}) classified typical representations
occurring in cuspidal representation of $\g{n}{F}$ for $n\geq 3$. It
turns out that there exists a unique typical representation occurring
in a cuspidal representation. For a general Bernstein component $s$,
typical representations may not be unique. In this article we classify
typical representations for a level zero Bernstein component~$s$.

We now describe the main result of this article.  Let $s$ be a level
zero Bernstein component and $(J_s, \lambda_s)$ (see section 3 for a
complete description of the type $(J_s, \lambda_s)$) be the
Bushnell--Kutzko type for the component $s$. We will prove that
\begin{theorem}
  Any typical representation $\tau_s$ for a level-zero Bernstein
  component $s$ occurs as a subrepresentation of
\begin{equation}\label{intro_equation_1}
  \ind_{J_s}^{\g{n}{\integers{F}}}(\lambda_s).
\end{equation}
\end{theorem}
In our analysis we will also obtain a certain multiplicity result on
the typical representations $\tau_s$ (see Corollary
\ref{levelzero_coro}).

This article is based on chapter $3$ of my thesis. In my thesis typical
representations are classified for several Bernstein components. I
would like to thank my thesis advisor Guy Henniart for suggesting this
problem and numerous discussions. I thank Corinne Blondel for pointing
out several corrections and improvements. I would like to express my
deep gratitude to the referee for very useful suggestions and
comments.

\section{Preliminaries}
\subsection{Basic notation}
Let $F$ be a non-Archimedean local field with ring of integers
$\integers{F}$, maximal ideal $\ideal{F}$ and a finite residue field
$k_F$. All our representations are on vector spaces over $\mathbb{C}$.

Let $G$ be a locally pro-finite group and $H$ be a closed subgroup of
$G$. For $(\tau, V)$ a smooth representation of $H$, we denote by
$\ind_{H}^{G}(\tau)$ the induced smooth representation of $G$ and by
c-$\ind_{H}^{G}(\tau)$ the compactly induced representation. When $G$
is the group of $F$-rational points of an algebraic reductive group,
the group $G$ is equipped with a locally profinite topology induced
from $F$. For $P$ the set of $F$-rational points of an $F$-parabolic
subgroup of $G$ and $\sigma$ a smooth representation of a Levi
subgroup $M$ of $P$, we denote by $i^G_P(\sigma)$ the normalized
parabolically-induced representation.

For any two groups $H_1$ and $H_2$ such that $H_2\subset H_1$ and
$\sigma$ a representation of $H_1$, we denote by $\res_{H_2}(\sigma)$
the restriction of $\sigma$ to $H_2$. We use $\boxtimes$ and $\otimes$
for the tensor product of representations of two different groups and
the same group respectively.  If $H_2$ is a subgroup of a group $H_1$,
$\tau$ is a representation of $H_2$ and $h\in H_1$ then we denote by
$^{h}\tau$ the representation of $hH_2h^{-1}$ given by
$h'\mapsto \tau(h^{-1}h'h)$ for all $h'\in hH_2h^{-1}$.

After recalling some general definitions we will restrict ourself to
the case where $G=\g{n}{F}$ and the following notation will be used:
We denote by $G_n$ the group $\g{n}{F}$ and by $K_n$ the maximal
compact subgroup $\g{n}{\integers{F}}$. Let $K_n(m)$ the principal
congruence subgroup of $\g{n}{\integers{F}}$ of level $m$.

Let $I=(n_1,n_2,n_3,\dots, n_r)$ be an ordered partition of a positive
integer~$n$. Let $R$ be an $\mathcal{O}_F$ algebra. Let $P_I(R)$ be
the group of invertible block upper triangular matrices of type
$(n_1,n_2,\dots, n_r)$ with entries in $R$. We denote by $M_I(R)$ and
$U_I(R)$ the subgroups of $P_I(R)$ consisting of block diagonal
matrices of type $I$ and the unipotent unipotent matrices of type $I$
respectively. We use the notation $P_I$, $M_I$ and $U_I$ for $P_I(F)$,
$M_I(F)$ and $U_I(F)$ respectively.  We call $P_I$ and $M_I$ the
standard parabolic subgroup and standard Levi subgroup of type $I$
respectively.

\subsection{Bernstein decomposition and typical representations}
Let $B(G)$ be the set of pairs $(M, \sigma)$ where $M$ is a
Levi subgroup of an $F$-parabolic subgroup $P$ of $G$ and $\sigma$ is
an irreducible supercuspidal representation of $M$. Recall that the inertial
equivalence relation on $B(G)$ is defined  by setting
$$(M_1, \sigma_1)\sim (M_2, \sigma_2)$$
if and only if there exist an element $g\in G$ and an unramified
character $\chi$ of $M_2$ such that $M_1=gM_2g^{-1}$ and
$\sigma_1^g\simeq \sigma_2\otimes\chi$. We denote by $\mathcal{B}_G$
the set of such equivalence classes called {\bf inertial classes} or
{\bf Bernstein components}. Any irreducible smooth representation
$\pi$ of $G$ occurs as a sub-representation of a parabolic induction
$i_{P}^{G}(\sigma)$ where $\sigma$ is an irreducible supercuspidal
representation of a Levi subgroup $M$ of $P$. The pair $(M, \sigma)$
is well determined up to $G$-conjugation. We call the class
$s=[M, \sigma]$ the {\bf inertial support of $\pi$}. We will use the
notation $I(\pi)$ for the inertial support of $\pi$.

Let $\mathcal{M}(G)$ be the category of all smooth representations of
$G$. For an inertial class $s=[M, \sigma]$ we denote by
$\mathcal{M}_s(G)$ the full sub-category consisting of smooth
representations all of whose irreducible sub-quotients appear in the
composition series of some $i_{P}^{G}(\sigma\otimes\chi)$, with $\chi$
an unramified character of $M$.  It is shown by Bernstein (see \cite
[VI.7.2, Theorem]{Renard-p-adicrepbook}) that the category
$\mathcal{M}(G)$ decomposes as a direct product of $\mathcal{M}_s(G)$
in particular every smooth representation can be written as a direct
sum of objects in the categories $\mathcal{M}_s(G)$. We denote by
$\mathcal{A}_G(s)$ the set of isomorphism classes of simple objects in
the category $\mathcal{M}_s(G)$. If $G=\g{n}{F}$ we use the notation
$\mathcal{A}_n(s)$ for $\mathcal{A}_G(s)$ and $\mathcal{B}_n$ for
$\mathcal{B}_{\g{n}{F}}$.

Given an irreducible smooth representation $\rho$ of a maximal compact
subgroup $K$ of $G$ the compact induction
$\pi:=\text{c-}\ind_{K}^{G}(\rho)$ is a finitely generated smooth
representation of $G$ and hence there exists an irreducible
$G$-quotient of $\pi$. By Frobenius reciprocity \cite[Proposition
2.5]{Yellowbook} we get that $\rho$ occurs in a smooth irreducible
representation of $G$. For a given inertial class, we are interested
in the representations $\rho$ of $K$ which only occur in irreducible
smooth representations with inertial support $s$.

\begin{definition}\label{definition_main}
  Let $s$ be an inertial class for $G$. An irreducible smooth
  representation $\tau$ of a maximal compact subgroup $K$ of $G$ is
  called {\bf $K$-typical representation for $s$} if, for any irreducible
  smooth representation
  $\pi$ of $G$, $\ho_{K}(\tau, \pi)\neq 0$ implies that  $\pi\in \mathcal{A}_{G}(s)$.
\end{definition}

In this article we will confine ourselves to the cases where
$G=\g{n}{F}=G_n$, $K=\g{n}{\integers{F}}=K_n$ and $n\geq 2$ and in
these cases we call a $K$-typical representation for $s$ a typical
representation for $s$.  An irreducible representation $\tau$ of $K_n$
is called \emph{atypical} if $\tau$ occurs in two irreducible smooth
representations $\pi_1$ and $\pi_2$ such that $I(\pi_1)\neq I(\pi_2)$.

For any component $s\in \mathcal{B}_n$, the existence of a typical
representation can be deduced from the theory of types developed by
Bushnell and Kutzko in the articles
\cite{Bushnell-kutzko-Semisimpletypes} and \cite{Orrangebook}.
Bushnell and Kutzko constructed a pair $(J_s, \lambda_s)$, which we
call a \emph{Bushnell--Kutzko type}, where $J_s$ is a compact open
subgroup of $\g{n}{F}$ and $\lambda_s$ is an irreducible
representation of $J_s$ such that for every irreducible smooth
representation $\pi$ of $G_n$,
$$\ho_{J_s}(\pi, \lambda_s)\neq 0\ \Leftrightarrow\ \pi \ \in \ \mathcal{A}_n(s).$$ 
The group $J_s$ can be arranged to be a subgroup of
$\g{n}{\integers{F}}$ by conjugating with an element of $\g{n}{F}$ and
hence we assume that $J_s\subset \g{n}{\integers{F}}$.  It follows
from Frobenius reciprocity that any irreducible sub-representation of
\begin{equation}\label{bushnell_max_induction}
\ind_{J_s}^{\g{n}{\integers{F}}}(\lambda_s)
\end{equation}
is a typical representation. The irreducible sub-representations of
(\ref{bushnell_max_induction}) are classified by Schneider and Zink in
\cite[Section 6, $T_{K,\lambda}$ functor]{SchneiderKtypes}.

For $s=[G_n, \sigma]$, Pa\v{s}k\={u}nas in 
\cite [Theorem 8.1]{Paskunas-uniqueness} showed that up to isomorphism
there exists a unique typical representation for $s$. More precisely,
\begin{theorem}[Pa\v{s}k\={u}nas]
  Let $n$ be a positive integer greater than one and $\sigma$ be an
  irreducible supercuspidal representation of $G_n$. Let
  $(J_s, \lambda_s)$ be a Bushnell--Kutzko type for the component
  $s=[G_n, \sigma]$ with $J_s\subset \g{n}{\integers{F}}$. The
  representation
$$\ind_{J_s}^{K_n}(\lambda_s)$$
is the unique typical representation for the component $[G_n, \sigma]$
and occurs with multiplicity one in $\sigma\otimes\chi$, for all
unramified characters $\chi$ of $G_n$.
\end{theorem}
We will consider the classification of typical representations for
components $[M, \sigma]$ where $M$ is a Levi subgroup of a proper
parabolic subgroup of $G_n$.

Let $s=[M, \sigma]$ be an inertial class of $G_n$. We will choose a
representative for $s$. Let $P$ be a parabolic subgroup with $M$ as
its Levi subgroup. There exists a $g\in G_n$ such that $gPg^{-1}=P_I$
for some ordered partition $I=(n_1,n_2,\dots,n_r)$ of $n$. The groups
$gMg^{-1}$ and $M_I$ are two Levi subgroups of $P_I$ hence we get an
$u\in \Rad{P_I}$ such that $ugM(ug)^{-1}=M_I$. This shows that there
exists an element $g'\in G_n$ such that $g'Mg'^{-1}= M_I$. Let $J$ be
a permutation of the ordered partition $(n_1,n_2,\dots,n_r)$. We can
choose a $g''\in G_n$ such that $M_I$ and $M_J$ are conjugate so the
two pairs $(M, \sigma)$ and $(M_J, \sigma^{g'g''})$ are inertially
equivalent. In certain cases it is convenient to choose a particular
permutation. For example in the proof of the main theorem in this
article we choose $J=(n_1', n_2',\dots n_r')$ such that
$n_i'\leq n_j'$ for all $i\leq j$. We denote by $\sigma_I$ and
$\sigma_J$ the representations $\sigma^{g'}$ and $\sigma^{g'g''}$
respectively and hence $s=[M_I, \sigma_I]=[M_J, \sigma_J]$.

Let $\tau$ be a typical representation for the component $s$. The
representation $\tau$ occurs as a $K_n$ sub-representation of a
$G_n$-irreducible smooth representation $\pi$ (see the reasoning given
in the paragraph above Definition \ref{definition_main}).  From the
above paragraph $\pi$ occurs in the composition series of
$i_{P_I}^{G_n}(\sigma_I)$ where $\sigma_I$ is a supercuspidal
representation of $M_I$.  Hence to classify typical representations we
fix a pair $(M_I, \sigma_I)\sim (M,\sigma)$ and examine the
$K_n$-irreducible sub-representations of
$$\res_{K_n}( i_{P_I}^{G_n}(\sigma_I)),$$
looking for possible typical representations for $s$.

By the Iwasawa decomposition $G_n=K_nP_I$ we get that
$$\res_{K_n}( i_{P_I}^{G_n}(\sigma_I))\simeq  \ind_{P_I\cap K_n}^{K_n}(\sigma_I).$$
We write $\sigma_I$ as $\boxtimes_{i=1}^{r}\sigma_i$ where $\sigma_i$
is a supercuspidal representation of $G_{n_i}$ for $1\leq i\leq r$.
We denote by $\tau_i$ the unique typical representation for the
component $[G_{n_i}, \sigma_i]$ for $1\leq i\leq r$ and let $\tau_I$
be the $M_I(\integers{F})$-representation
$\boxtimes_{i=1}^{r}\tau_i$. Will Conley observed in his thesis (see
\cite[Theorem 2.28]{conley_thesis}) that the representation
$$\ind_{P_I\cap K_n}^{K_n}(\tau_I)$$
admits a complement in $\ind_{P_I\cap K_n}^{K_n}(\sigma_I)$ whose
irreducible sub-representations are atypical for $s$. We prove a mild
generalisation which will be used later in proofs by induction.

Let $t_i=[M_i, \lambda_i]$ be a Bernstein component of $G_{n_i}$
for $1\leq i\leq r$.  Let $\sigma_i$ be a smooth representation from
$\mathcal{M}_{t_i}({G_{n_i}})$. We suppose
$$\res_{K_{n_i}}\sigma_i=\tau_i^0\oplus\tau_i^1$$ for
$1\leq i\leq r$ such that irreducible subrepresentations of
$\tau_i^1$ are atypical.  We denote by $t$ the Bernstein component
$$[M_1\times M_2 \times \dots \times M_r,
\lambda_1\boxtimes\lambda_2\boxtimes\dots\boxtimes\lambda_r]$$ of
$G_n$. 
The component $t$ is independent of the choice of representatives 
$(M_i, \lambda_i)$. Let $\tau_I^0=\boxtimes_{i=1}^{r}\tau_i^0$ and 
$\sigma_I=\boxtimes_{i=1}^{r}(\sigma_i)$.

\begin{proposition}\label{proposition_prelim_1}
The representation 
$$\ind_{P_I\cap K_n}^{K_n}(\tau_I^0)$$
admits a complement in $\res_{K_n}i_{P_I}^{G_n}(\sigma_I)$ with all its
irreducible subrepresentations atypical.
\end{proposition}

\begin{proof}
  Any $K_n$-irreducible subrepresentation of
  $\res_{K_n}i_{P_I}^{G_n}(\sigma_I)$ occurs as a subrepresentation of
\begin{equation}\label{equation_prelim_1}
\ind_{P_I\cap K_n}^{K_n}(\boxtimes_{i=1}^{r}\gamma_i)
\end{equation}
where $\gamma_i$ is a $K_{n_i}$-irreducible subrepresentation of
$\sigma_i$. Suppose there exists $N\le r$ such that $\gamma_N$ occurs
in $\tau_N^1$. Thus there exists a component
$t_N'\in \mathcal{B}_{n_N}$ such that $t_N'$ is equal to
$ [M_N', \lambda_N']\neq t_N$ and $\gamma_N$ occurs in the restriction
$\res_{K_N}i_{P_N'}^{\g{n_N}{F}}(\lambda_N')$.  Hence the
representation (\ref{equation_prelim_1}) occurs as a
$K_n$-subrepresentation of
$$i_{P_I}^{G_n}\{i_{P_1}^{\g{n_1}{F}}
(\lambda_1)\boxtimes\dots\boxtimes i_{P_N'}^{\g{n_N}{F}}(\lambda_N')
\boxtimes\dots\boxtimes i_{P_r}^{\g{n_r}{F}}(\lambda_r)\}$$
The inertial support $t'$ of the above representation is
$$[M_1\times\dots\times M_{N}'\times\dots\times M_r, 
\lambda_1\boxtimes\dots\boxtimes\lambda_N'\boxtimes\dots\boxtimes\lambda_r].$$
We may assume that $M_i$ is a standard Levi subgroup for
$1\leq i\leq r$. Now
$$[M_N=\prod_{j=1}^{p}\g{m_j}{F},
\lambda_N=\boxtimes_{j=1}^{p}\zeta_j]\neq 
[M_N'=\prod_{j=1}^{p'}\g{m_j'}{F}, \lambda_N'=\boxtimes_{j=1}^{p'}\zeta_j']$$
implies that there exists a cuspidal component $[\g{m_k}{F},
\zeta_{k}]$ occurring 
in the multi-set 
$$\{[\g{m_1}{F}, \zeta_1], [\g{m_2}{F}, \zeta_2], \dots, [\g{m_p}{F}, \zeta_p]\}$$
which has a different multiplicity in 
$$\{[\g{m_1'}{F}, \zeta_{1}'], [\g{m_2'}{F}, \zeta_{2}'],\dots, [\g{m_{p'}}{F}, \zeta_{p'}']\}.$$
Adding cuspidal components with the same multiplicity to the above two
multi-sets cannot make the multiplicities of the component
$[\g{k}{F}, \zeta_k]$ the same. This shows that $t'\neq t$ and hence
the desired complement is the direct sum of the representations as in
(\ref{equation_prelim_1}) such that $\gamma_i$ occur in $\tau_i^1$ for
some $i\in \{1,2,\dots,r\}$.
\end{proof}

\begin{lemma}\label{lemma_prelim_1}
  Let $t_i=[G_{n_i}, \sigma_i]$ be a Bernstein component for
  $G_{n_i}$ and $\tau_i$ be a typical representation for $t_i$ and let
  $\tau_I$ be the representation
  $\tau_1\boxtimes\tau_2\boxtimes\dots\boxtimes\tau_s$.  The
  representation
$$\ind_{P_I\cap K_n}^{K_n}(\tau_I)$$
admits a complement in $\res_{K_n}i_{P_I}^{G_n}(\sigma_I)$ whose
irreducible sub-representations are atypical.
\end{lemma}

\begin{proof}
  We use the uniqueness of typical representations for supercuspidal representations
  (see \cite{Paskunas-uniqueness}) to decompose
  $\res_{K_{n_i}}\sigma_i$ as $\tau_i\oplus \tau_i^1$ such
  that irreducible sub-representations of $\tau^1_i$ are atypical.
  The lemma follows as a consequence of Proposition
  \ref{proposition_prelim_1}.
\end{proof}

Given a component $s=[M_I, \sigma_I]$ of $G_n$ the above lemma
shows that typical representations only occur as sub-representations
of
$$\ind_{P_I\cap K_n}^{K_n}(\tau_I).$$
The above representation is still an infinite dimensional
representation of the compact group $K_n$. We write
the above representation as an increasing union of finite-dimensional
representations.

Let $\{H_i\}_{i\geq 1}$ be a decreasing sequence of compact open
subgroups of the maximal compact subgroup $K_n$. Let
$\bar{U}_I$ be the unipotent radical of the opposite parabolic
subgroup $\bar{P}_I$ of $P_I$ with respect to the Levi subgroup
$M_I$. We assume that $H_i$ has an Iwahori decomposition with
respect to the parabolic subgroup $P_I$ and Levi subgroup $M_I$ for
all $i\geq 1$ i.e. the product map
$$(H_i\cap\bar{U_I})\times(H_i\cap M_I)\times(H_i\cap U_I)\rightarrow H_i$$
is a homeomorphism for any ordering of the factors on the left hand
side and that $\bigcap_{i\geq1}H_i=K_n\cap P_I$.  Let
$\tau$ be a finite dimensional smooth representation of the group
$M_I(\mathcal{O}_F)$. We assume that $\tau$ extends to a
representation of $H_i$ for all $i\geq 1$ such that $H_i\cap U_I$ and
$H_i\cap \bar{U}_I$ are contained in the kernel of $\tau$.  By
definition the representation $\ind_{H_i}^{K_n}(\tau)$
is contained in
$\ind_{K_n\cap P_I}^{K_n}(\tau)$.

\begin{lemma}\label{lemma_prelim_2}
  The union of the
  representations $$\ind_{H_i}^{K_n}(\tau)$$ for all
  $i\geq 1$ is equal to the representation
$$\ind_{K_n\cap P_I}^{K_n}(\tau).$$
\end{lemma}

\begin{proof}
  Let $W$ be the underlying space for the representation $\tau$. Any
  element $f$ in the space
$$\ind_{K_n\cap P_I}^{K_n}(\tau)$$
is a function $f:K_n\rightarrow W$ such that 
\begin{enumerate}
\item $f(pk)=\tau(p)f(k)$ for all $p\in K_n\cap P_I$ and $k\in K_n,$
\item here exists a positive integer $m$ (depending on $f$) such that
  $f(gk)=f(g)$ for all $k\in K_n(m)$ and $g\in K_n$.
\end{enumerate}
Now there exists a positive integer $i$ such that
$H_i\cap \bar{U}_I \subset K_n(m)$. For such a choice of $i$ and
$h\in H_i$ write $h=h^{-}h^{+}$ where
$h^{+}\in {K_n}\cap P$, $h^{-}\in H_i\cap \bar{U}_I$ which
we can do so by Iwahori decomposition of $H_i$.  We observe that
$f(hk)=f(h^{-}h^{+}k)=f(h^{+}k(h^{+}k)^{-1}h^{-}(h^{+}k))=f(h^{+}k)=\tau(h^{+})f(k)$
(since $(h^{+}k)^{-1}h^{-}(h^{+}k)\in K_n(m)$).  Hence
$f\in \ind_{H_i}^{K_n}(\tau)$.
\end{proof}

We shall need the following technical lemma for frequent
reference. Let $P$ be any parabolic subgroup of $G_n$ with a
Levi subgroup $M$ and $U$ be the unipotent radical of $P$. Let $J_1$
and $J_2$ be two compact open subgroups of $K_n$ such that $J_1$
contains $J_2$.  Suppose $J_1$ and $J_2$ both satisfy Iwahori
decomposition with respect to the Levi subgroup $M$. With
$J_1\cap U=J_2\cap U$ and $J_1\cap \bar{U}=J_2\cap \bar{U}$. Let
$\lambda$ be an irreducible smooth representation of $J_2$ which
admits an Iwahori decomposition i.e. $J_2\cap U$ and $J_2\cap \bar{U}$
are contained in the kernel of $\lambda$.

\begin{lemma}\label{lemma_prelim_3}
  The representation $\ind_{J_2}^{J_1}(\lambda)$ is the extension of
  the representation $\ind_{J_2\cap M}^{J_1\cap M}(\lambda)$ such that
  $J_1\cap U$ and $J_1\cap \bar{U}$ are contained in the kernel of the
  extension.
\end{lemma}

\begin{proof}
  From the Iwahori decomposition we get that $(J_1\cap M)J_2=J_1$ and
  from the Mackey decomposition we get that
$$\res_{J_1\cap M}\ind^{J_1}_{J_2}(\lambda)\simeq\ind_{J_2\cap M}^{J_1\cap M}(\lambda).$$
We now verify that $J_1\cap U$ and $J_1\cap \bar{U}$ act trivially on
$\ind_{J_2}^{J_1}(\lambda)$. Observe that
$$\res_{J_1\cap P}\ind^{J_1}_{J_2}(\lambda)\simeq\ind_{J_2\cap P}^{J_1\cap P}(\lambda).$$
Since the double coset representatives for 
$$\dfrac{J_1\cap P}{J_2\cap P}$$
can be chosen from $M\cap J_1$ the group $J_1\cap U$ acts trivially on
$\ind_{J_2}^{J_1}(\lambda)$. Similarly $J_1\cap \bar{U}$ acts
trivially on $\ind_{J_2}^{J_1}(\lambda)$. This concludes the lemma.
\end{proof}

\begin{lemma}\label{lemma_prelim_4}
  Let $G$ be the $F$-rational points of an algebraic reductive group and
  $\chi$ be a character of $G$. Let $\tau$ be a $K$-typical
  representation for the component $s=[M, \sigma]$.  The representation
  $\tau\otimes\chi$ is a typical representation for the component
  $[M, \sigma\otimes\chi]$.
\end{lemma}

\begin{proof}
  Let $\ho_{K}(\tau\otimes\chi, \pi)\neq 0$ for some irreducible
  smooth representation $\pi$ of $G$. We now have
  $\ho_{K}(\tau, \pi\otimes\chi^{-1})\neq 0$. This implies that
  $\pi\otimes\chi^{-1}$ occurs in the composition series of
$$i_{P}^{G}(\sigma\otimes\eta)$$
for some parabolic subgroup $P$ containing $M$ as a Levi subgroup
and $\eta$ an unramified character of $M$. 
Now $\pi$ occurs in the
composition series for the representation
$$i_P^G(\sigma\otimes\chi\otimes\eta)$$
hence $\tau\otimes\chi$ is a $K$-typical representation for the
component $[M, \sigma\otimes\chi]$.
\end{proof}

\section{Level-Zero Bernstein components}
\begin{definition}
  Let $I=(n_1,n_2,\dots,n_r)$ be an ordered partition of $n$. An
  inertial class $s=[M_I, \boxtimes_{i=1}^{r}\sigma_i]$ is called a
  level-zero inertial class if the $K_{n_i}(1)$ invariants of
  $\sigma_i$ is non trivial, for $1\leq i\leq r$.  
\end{definition}

We fix a level-zero inertial class $s=[M_I, \sigma_I]$. The subgroup
$K_{n_i}$ acts on the $K_{n_i}(1)$ invariants of $\sigma_i$ and
$\tau_i$ be this representation of $K_{n_i}$ on $K_{n_i}(1)$
invariants of $\sigma_i$. The representation is $\tau_i$ is the
inflation of a cuspidal representation of $\g{n_i}{k_F}$. {\bf The
  pair $(K_{n_i}, \tau_i)$ is the Bushnell--Kutzko type
  for the inertial class $[G_{n_i}, \sigma_i]$}.  

Let $m$ be a positive integer and $P_I(m)$ be the inverse image of
$P_I(\mathcal{O}_F/\mathfrak{P}_F^m)$ under the mod-$\mathfrak{P}_F^m$
reduction map
$$\pi_m:K_n\rightarrow\g{n}{\mathcal{O}_F/\mathfrak{P}_F^m}.$$ 
The representation $\boxtimes_{i=1}^{r}\tau_i$ of $M_I(k_F)$ can be
viewed as a representation of $P_I(k_F)$ by inflation via the quotient
map
$$P_I(k_F)\rightarrow P_I(k_F)/U_I(k_F)\simeq M_I(k_F).$$
The representation $\boxtimes_{i=1}^{r}\tau_i$ of $P_I(k_F)$ is also
 a representation of $P_I(1)$ by inflation via the map $\pi_1$. We
note that $P_I(1)\cap U_I$ and $P_I(1)\cap \bar{U}_I$ are contained in
the kernel of this extension.  {\bf The pair $(P_I(1), \tau_I)$ is the
Bushnell--Kutzko type for the component $s$} (see
\cite [Section 8.3.1]{Bushnell-kutzko-Semisimpletypes}). The
irreducible sub-representations of
$$\ind_{P_I(1)}^{K_n}(\tau_I)$$
are thus typical for $s$.

We note that the groups $P_I(m)$ have Iwahori decomposition with
respect to $P_I$ and $M_I$.  The representation $\tau_I$ of
$M_I(\integers{F})$ extends to a representation of $P_I(m)$ such that
$P_I(m)\cap U_I$ and $P_i(m)\cap \bar{U}_I$ are contained in the
kernel of the extension. This shows that the sequence of groups
$\{P_I(m)\ | \ m\geq 1\}$ and $\tau_I$ satisfy the hypothesis for the
groups $\{H_m\ | m\geq 1\}$ and $\tau$ in Lemma \ref{lemma_prelim_2}
hence we have the isomorphism
$$\bigcup_{m\geq 1} \ind_{P_I(m)}^{K_n}(\tau_I)\simeq \ind_{P_I\cap K_n}^{K_n}(\tau_I).$$
We recall that the Lemma \ref{lemma_prelim_1} shows that typical
representations for the component $s$ can only occur in the above
representation.

Using Frobenius reciprocity we get that the representation $\tau_I$
occurs in $\ind_{P_I(m)}^{P_I(1)}(\tau_I)$ with multiplicity one. Let
$m\geq 1$ and $U_m^0(\tau_I)$ be the $P_I(1)$-stable complement of the
representation $\tau_I$ in $\ind_{P_I(m)}^{P_I(1)}(\tau_I)$. Let
$U_m(\tau_I)$ be the representation
$$\ind_{P_I(1)}^{K_n}(U_m^0(\tau_I)).$$
We note that 
$$\ind_{P_I(1)}^{K_n}(\tau_I)\oplus U_m(\tau_I)\simeq \ind_{P_I(m)}^{K_n}(\tau_I) $$
We will show that irreducible sub-representations of $U_m(\tau_I)$ are
atypical.

\begin{theorem}[Main]\label{theorem_level_1}
  Let $m\geq 1$.  The $K_n$-irreducible
  subrepresentations of $U_m(\tau_I)$ are atypical.
\end{theorem}

Using this, the classification of typical representations for the inertial class
$s$ is given by the following corollary.

\begin{corollary}\label{levelzero_coro}
  The irreducible sub-representations of
  $\ind_{P_I(1)}^{K_n}(\tau_I)$ are precisely the
  typical representations for the level-zero inertial class
  $[M_I, \sigma_I]$. Moreover if $\Gamma$ is a typical representation
  then
$$\dim_{\mathbb{C}}\ \ho_{K_n}(\Gamma, \ind_{P_I(1)}^{K_n}(\tau_I))= 
\dim_{\mathbb{C}}\ \ho_{K_n}(\Gamma, i_{P_I}^{G_n}(\sigma_I)).$$
\end{corollary}

\begin{proof}
  Given a typical representation $\Gamma$ for the inertial class $s$,
  the theorem shows that $\Gamma$ is a sub-representation of
  $\ind_{P_I(1)}^{K_n}(\tau_I)$ and the multiplicity
  formula follows from Lemma \ref{lemma_prelim_1} and the above
  theorem. Conversely if $\Gamma$ is a sub-representation of
  $\ind_{P_I(1)}^{K_n}(\tau_I)$ then, by Frobenius
  reciprocity, we get that $\ho_{P_I(1)}(\tau_I, \Gamma)\neq 0$. If
  $\Gamma$ is contained as a $K_n$-irreducible
  sub-representation in an irreducible smooth representation $\pi$ of
  $G_n$ then the restriction of $\pi$ to $P_I(1)$ contains the
  representation $\tau_I$.  The pair $(P_I(1), \tau_I)$ is the
  Bushnell--Kutzko type for the inertial class $s=[M_I, \sigma_I]$
  hence the inertial support of $\pi$ is $s$. Hence $\Gamma$ is a
  typical representation and this proves the corollary.
\end{proof}

\subsection{Decomposition of an auxiliary representation}
We will need a few lemmas regarding the splitting of a certain
representation for the proof of the main theorem. Let $I$ be the
ordered partition $(n_1, n_2,\dots, n_r)$ of the positive integer $n$
as fixed at the beginning of this chapter.  {\bf Until the beginning
  of the section \ref{section_chapter1_1} we assume that $r>1$} in
other words $M_I$ is a proper Levi subgroup.  We denote by $I'$ the
ordered partition $(n_1,n_2,\dots,n_{r-1})$ of $n-n_r$. Let $m$ be a
positive integer and $P_I(1,m)$ be the following set
$$\left\{ \begin{pmatrix}A&B\\\varpi_F^mC&D\end{pmatrix}| 
A\in P_{I'}(1); B^{tr}, C\in M_{n_r\times(n-n_r)}(\integers{F}); D\in K_{n_r}\right\}.$$
Here $tr$ denotes transpose. Note that $P_I(1,1)=P_I(1)$.

\begin{lemma}
The set $P_I(1,m)$ is a subgroup of $P_I(1)$. 
\end{lemma}

\begin{proof}
  The group $K_n$ acts on the set of lattices of $F^n$
  contained in the lattice $\integers{F}^n$. If $r-1=1$ the set
  $P_I(1,m)$ is the $K_n$-stabilizer of the lattice
  $(\integers{F})^{n_1}\oplus(\varpi_F^m\integers{F})^{n_2}$. In the
  case $r-1>1$ the set $P_I(1,m)$ is the
  $K_n$-stabilizer of the set of  lattices $\{L_k \mid 1<k\le r-1\}$ 
  defined by:
$$L_k=(\integers{F})^{n_1}\oplus \dots\oplus (\integers{F})^{n_{k-1}}
\oplus (\varpi_F\integers{F})^{n_k}\oplus\dots\oplus
(\varpi_F\integers{F})^{n_{r-1}}
\oplus(\varpi_F^m\integers{F})^{n_r}.$$
This shows that $P_I(1,m)$ is a subgroup and is contained in $P_I(1)$
from the definition.
\end{proof}

The structure of the representation 
$$\ind_{P_I(1,m+1)}^{P_I(1,m)}(\id)$$
will be used in the proof of the main theorem. Using Clifford theory
we decompose the above representation. Let $K_I(m)$ be the group
$K_n(m)U_{(n-n_r, n_r)}(\integers{F})$. We note that this group only
depends on $n$ and $n_r$, rather than the whole partition $I$. 

\begin{lemma}
  The group $K_I(m)$ is a normal subgroup of $P_I(1,m)$ and
  $K_I(m)\cap P_I(1,m+1)$ is a normal subgroup of $K_I(m)$.
\end{lemma}

\begin{proof}
  The groups $K_I(m)$ and $P_I(1,m)$ satisfy Iwahori decomposition
  with respect to $U_{(n-n_r,n_r)}$, $\bar{U}_{(n-n_r,n_r)}$ and
  $M_{(n-n_r,n_r)}$. We also note that
$$K_I(m)\cap U_{(n-n_r,n_r)}=P_I(1,m)\cap U_{(n-n_r,n_r)}$$
 and 
 $$K_I(m)\cap \bar{U}_{(n-n_r,n_r)}=P_I(1,m)\cap \bar{U}_{(n-n_r,n_r)}.$$
 Hence $P_I(1,m)\cap U_{(n-n_r,n_r)}$ and
 $P_I(1,m)\cap \bar{U}_{(n-n_r,n_r)}$ normalize $K_I(m)$. Since
 $K_I(m)$ is a product of the group $K_n(m)$ and
 $U_{(n-n_r)}(\integers{F})$ the group\\
 $P_I(1,m)\cap M_{(n-n_r,n_r)}$ normalizes the group $K_I(m)$. This
 shows the first part.

 Notice that $K_I(m)\cap U_{(n-n_r,n_r)}$ is equal to
 $K_I(m)\cap P_I(1,m+1)\cap U_{(n-n_r,n_r)}$ and
 $K_I(m)\cap M_{(n-n_r,n_r)}$ is equal to
 $K_I(m)\cap P_I(1,m+1)\cap M_{(n-n_r,n_r)}$ hence it is enough to
 check that $K_I(m)\cap \bar{U}_{(n-n_r,n_r)}$ normalizes the group
 $K_I(m)\cap P_I(1,m+1)$.  Since
 $K_I(m)\cap P_I(1,m+1)\cap \bar{U}_{(n-n_r,n_r)}$ is abelian and is
 contained in $K_I(m)\cap \bar{U}_{(n-n_r,n_r)}$ hence we need to
 check that $u^{-}j(u^{-})^{-1}$ and $u^{-}u^{+}(u^{-})^{-1}$ are
 contained in $K_I(m)\cap P_I(1,m+1)$ for all $u^{-}$, $j$ and $u^{+}$
 in
\begin{align*}
   & K_I(m)\cap \bar{U}_{(n-n_r,n_r)}, \\ 
   & K_I(m)\cap P_I(1,m+1)\cap M_{(n-n_r,n_r)} \ \text{and} \\  
    & K_I(m)\cap P_I(1,m+1)\cap U_{(n-n_r,n_r)}=U_{(n-n_r,n_r)}(\integers{F}) 
 \end{align*} 
 respectively.  Let $u^{+}$, $u^{-}$ and $j$ be three elements from
 $U_{n-n_r, n_r}(\integers{F})$, $K_I(m)\cap \bar{U}_{(n-n_r,n_r)}$
 and $K_I(m)\cap P_I(1,m+1)\cap M_{(n-n_r, n_r)}$ respectively. We
 write them in their block form as:
$$u^{+}=\begin{pmatrix}1_{n-n_r} &B\\0&1_{n_r}\end{pmatrix}$$
where $B\in M_{(n-n_r)\times n_r}(\integers{F})$, 
$$u^{-}=\begin{pmatrix}1_{n-n_r} &0\\\varpi_F^mC&1_{n_r}\end{pmatrix}$$ 
where $C\in M_{n_r\times (n-n_r)}(\integers{F})$  and 
$$j=\begin{pmatrix}J_1& 0\\ 0& J_2\end{pmatrix}.$$
We observe that $u^{-}j(u^{-})^{-1}=j\{j^{-1}u^{-}j(u^{-})^{-1}\}$
and the commutator\\
 $\{j^{-1}u^{-}j(u^{-})^{-1}\}$ in its block form is as follows: 
$$\begin{pmatrix}
  1_{n-n_r} && 0\\
  J_2^{-1}(\varpi_F^mCJ_1^{-1}-\varpi_F^mC) && 1_{n_r}
 \end{pmatrix}.$$
 We note that $J_2\in K_{n_r}(m)$ and $J_1\in K_{n-n_r}(m)$ hence
 $J_2^{-1}(\varpi_F^mCJ_1^{-1}-\varpi_F^mC)$ belongs to
 $$\varpi_F^{m+1}M_{(n-n_r)\times n_r}(\integers{F}).$$
 This shows that
$$\{j^{-1}u^{-}j(u^{-})^{-1}\}\in K_I(m)\cap P_I(m+1)$$
 Now the element $(u^{-})u^{+}(u^{-})^{-1}$ is of the form 
\begin{equation}\label{equation_level_2}
\begin{pmatrix}
1_{n-n_r}-\varpi_F^mBC&&B\\
-\varpi_F^{2m}CBC&& 1_{n_r}+\varpi_F^{m}CB
\end{pmatrix}.
\end{equation}
Since $2m\geq m+1$ the matrix in (\ref{equation_level_2}) is contained
in the group $K_{I}(m)\cap P_I(1,m+1)$.
\end{proof}

We now observe that $K_{I}(m)P_I(1,m+1)=P_I(1,m)$. From Mackey decomposition we get that
$$\res_{K_{I}(m)}\ind_{P_I(1,m+1)}^{P_I(1,m)}(\id)\simeq \ind_{K_{I}(m)\cap P_I(1,m+1)}^{K_{I}(m)}(\id).$$
Hence the above restriction decomposes into  a direct sum of representations of the group 
\begin{equation}\label{equation_level_2.5}
\dfrac{K_{I}(m)}{K_{I}(m)\cap P_I(1,m+1)}.
\end{equation}
The inclusion map of $K_I(m)\cap \bar{U}_{(n-n_r,n_r)}$ in $K_I(m)$
induces the natural homomorphism
$$\tilde{\theta}_{I}:\ \dfrac{K_I(m)\cap
  \bar{U}_{(n-n_r,n_r)}}{P_I(1,m+1)\cap  
\bar{U}_{(n-n_r,n_r)}}\rightarrow \dfrac{K_I(m)}{K_I(m)\cap P_I(1,m+1)}.$$

\begin{lemma}\label{lemma_level_3}
  The map $\tilde{\theta}_I$ is an $M_{(n-n_r, n_r)}\cap P_I(1,m)$
  equivariant isomorphism.
\end{lemma}

\begin{proof}
  The map is clearly injective and surjectivity follows from the
  Iwahori decomposition of $K_I(m)$ with respect to the Levi subgroup
  $M_I$. The inclusion of $K_I(m)\cap \bar{U}_{(n-n_r,n_r)}$ in
  $K_I(m)$ is an $M_{n-n_r, n_r}\cap P_I(1,m)$ equivariant map.
\end{proof}

Let $u^{-}$ be an element of the group
$K_I(m)\cap \bar{U}_{(n-n_r,n_r)}$ and its block form be given by
$$\begin{pmatrix}1_{(n-n_r, n_r)}&0\\U^{-}&1_{n_r}\end{pmatrix}.$$
The map $u^{-}\mapsto \varpi_F^{-m}U^{-}$ induces an isomorphism between the groups
$$K_I(m)\cap \bar{U}_{(n-n_r, n_r)}$$ 
and $M_{n_r\times (n-n_r)}(\integers{F})$. Let $\bar{U^{-}}$ be the image
of $U^{-}$ in the mod-$\ideal{F}$ reduction of
$M_{n_r\times (n-n_r)}(\integers{F})$.  The map
$u^{-}\mapsto \overline{\varpi_F^{-m}U^{-}}$ induces an isomorphism of the
quotient (\ref{equation_level_2.5}) with the group of matrices
$M_{n_r\times(n-n_r)}(k_F)$. We note that 
$M_{(n-n_r, n_r)}(\integers{F})=K_{n-n_r}\times
K_{n_r}$ acts on the group $M_{n_r\times(n-n_r)}(k_F)$ through its
mod-$\ideal{F}$ reduction $\g{n-n_r}{k_F}\times \g{n_r}{k_F}$, the
action is given by $(g_1, g_2)U=g_2Ug_1^{-1}$ for all
$g_1$ in $\g{n-n_r}{k_F}$, $g_2$ in  $\g{n_r}{k_F}$ and
$U$ in $M_{n_r\times(n-n_r)}(k_F)$. The map
$u^{-}\mapsto \overline{\varpi_F^{-m}U^{-}}$ is hence an
$M_{(n-n_r, n_r)}(\integers{F})$-equivariant map between the quotient
(\ref{equation_level_2.5}) and $M_{n_r\times(n-n_r)}(k_F)$. Moreover
the action of $M_{(n-n_r, n_r)}(\integers{F})$ factors through its
quotient $M_{(n-n_r, n_r}(k_F)$.

The space $M_{n\times m}(k_F)$ is equipped with an action of
$G:=\g{m}{k_F}\times \g{n}{k_F}$ given  by
$(g_1, g_2)U=g_2Ug_1^{-1}$. We also have a $G$ action on the set of
matrices $M_{m\times n}(k_F)$ by setting
$(g_1,g_2)V=g_1Vg_2^{-1}$. Let $\psi$ be a non-trivial character of
the additive group $k_F$. We define a pairing $B$ between
$M_{m\times n}(k_F)$ and $M_{n\times m}(k_F)$ by
defining $B(V,U)=\psi\circ\tr(VU)$. Let $T$ be the map from
$M_{m\times n}(k_F)$ and
$M_{n\times m}(k_F)^{\wedge}$ defined by
$$T(V)(U)=B(V,U).$$

\begin{lemma}\label{lemma_level_3.5}
 The map $T$ is a $G$-equivariant isomorphism.
\end{lemma}

\begin{proof}
That the map $T$ is $G$ equivariant can be verified from the identity
$$(g_1,g_2)T(V)(U)=\psi\circ\tr(Vg_2^{-1}Ug_1)=
\psi\circ\tr(g_1Vg_2^{-1}U)=T((g_1,g_2)V)(U).$$ 
It remains to show that $B$ is non-degenerate. Let $V_{ij}$ ($U_{ij}$)
be a matrix whose $ij$-th entry is $v_{ij}$ ($u_{ij}$) and all other
entries are zero. We observe that $B(U_{ij},
V_{ij})$ is equal to $\psi(u_{ij}v_{ij})$. This shows that $B$ is non-degenerate. 
\end{proof}

The above two lemmas gives an $M_{(n-n_r, n_r)}\cap P_I(1,m)$
 equivariant isomorphism  
\begin{equation}\label{equation_level_2.3}
\theta_I : \left\{\dfrac{K_{I}(m)}{K_{I}(m)\cap
    P_I(1,m+1)}\}\right\}^{\wedge} \rightarrow  M_{(n-n_r)\times n_r}(k_F).
\end{equation}

Since the group $K_I(m)$ is a normal subgroup of $P_I(1,m)$, we have
an action of this group $P_I(1,m)$ on the set of characters of the
abelian group
$$\dfrac{K_{I}(m)}{K_{I}(m)\cap P_I(1,m+1)}.$$
If $\eta$ is one such character we denote by $Z(\eta)$ the
$P_{I}(1,m)$-stabilizer of this character $\eta$. Clifford theory now
gives the decomposition
$$\ind_{P_I(1,m+1)}^{P_I(1,m)}(\id)\simeq \bigoplus_{\eta}\ind_{Z(\eta)}^{P_I(1,m)}(U_{\eta})$$
where $\eta$ runs over a set of representatives for the orbits under
the action of $P_I(1,m)$ and $U_{\eta}$ is some irreducible
representation of the group $Z(\eta)$.  We also note that
$Z(\id)=P_I(1,m)$ and the identity character occurs with 
multiplicity one (which follows from Frobenius reciprocity) and hence
\begin{equation}\label{equation_level_3}
  \ind_{P_I(1,m+1)}^{P_I(1,m)}(\id)\simeq \id\oplus 
  \bigoplus_{\eta\neq\id}\ind_{Z(\eta)}^{P_I(1,m)}(U_{\eta}).
\end{equation}

Observe that
$$Z(\eta)=(Z(\eta)\cap M_{(n-n_r,n_r)})K_I(m).$$
Let $\theta_I(\eta)=A$. Since $\theta_I$ is  $M_{(n-n_r, n_r)}\cap
P_I(1,m)$ equivariant we get that
$$Z(\eta)\cap M_{(n-n_r,n_r)}=Z_{M_{(n-n_r, n_r)}\cap P_I(1,m)}(A)$$
for some matrix $A$ in $M_{(n-n_r)\times n_r}(k_F)$. The group
$M_{(n-n_r, n_r)}\cap P_I(1,m)$ acts on the group of matrices
$M_{(n-n_r, n_r)}(k_F)$ through its mod-$\ideal{F}$ reduction. The
mod-$\ideal{F}$ reduction of the group $P_I(1,m)\cap M_{(n-n_r,n_r)}$
is equal to the group $P_{I'}(k_F)\times\g{n_r}{k_F}$.  In the
next lemma we will bound the mod $\ideal{F}$ reduction of the group
$Z(\eta)\cap M_I$ for the proof of the Main theorem.  Let
$\mathcal{O}_{A}$ be an orbit for the action of
$P_{I'}(k_F)\times\g{n_r}{k_F}$ on the set of matrices
$M_{(n-n_r)\times n_r}(k_F)$.  Let $p_j$ be the projection onto $j^{th}$
factor of $M_I(k_F)=\prod_{i=1}^{r}\g{n_i}{k_F}$.

\begin{lemma}\label{lemma_level_4}
Let $\mathcal{O}_A$ be an orbit consisting of non-zero matrices in 
$$M_{(n-n_r)\times n_r}(k_F).$$
 We can choose a representative $A$ such that the 
$P_{I'}(k_F)\times\g{n_r}{k_F}$-stabilizer of $A$, 
$$Z_{P_{I'}(k_F)\times\g{n_r}{k_F}}(A)$$ 
satisfies one of the following conditions.  
\begin{enumerate}
\item There exists a positive integer $j$, $j\leq r$ such that the
  image of
$$p_j:Z_{P_{I'}(k_F)\times\g{n_r}{k_F}}(A)\cap M_I(k_F)\rightarrow \g{n_j}{k_F}$$
is contained in a proper parabolic subgroup of $\g{n_j}{k_F}$.
\item There exists an $i$ with $1\leq i\leq r-1$ such that
  $p_i(g)=p_r(g)$ for all $g$ in 
$$ Z_{P_{I'}(k_F)\times\g{n_r}{k_F}}(A)\cap M_I(k_F).$$
\end{enumerate}
\end{lemma}

\begin{proof}
  Let $A=[U_{1},U_{2},\dots,U_{(r-1)}]^{tr}$ be the block form
  ($U_{k}$ is a matrix of size $n_r\times n_k$ for $1\leq k\leq r-1$)
  of a representative $m$ for an orbit $\mathcal{O}_m$ consisting of
  non-zero matrices.  If
  $((M_{ij}),B)\in Z_{P_{I'}(k_F)\times\g{n_r}{k_F}}(A)$ ($M_{ii}$ is
  a matrix of size $n_i\times n_i$) then we have
\begin{equation}\label{proper}
(M_{ij})[U_{1},U_{2},...,U_{(r-1)}]^{tr}=[U_{1},U_{2},...,U_{(r-1)}]^{tr}B.
\end{equation}
Since $(M_{ij})\in P_{I'}(k_F)$, we have $M_{ij}=0$ for all $i>j$. Let
$l\leq r-1$ be the maximal positive integer such that $U_{l}$
is non-zero and such an $l$
exists since $m\neq 0$.  From (\ref{proper}) we get that
$M_{ll}U_{l}^{tr}t=U_{l}^{tr}B$ where. There exist matrices
$P \in \g{n_r}{k_F}$ and $Q\in \g{n_l}{k_F}$ such that $PU_l^{tr}Q$ is a
matrix of the form
\begin{equation}\label{equation_level_6}
\begin{pmatrix}1_t&0\\0&0\end{pmatrix}
\end{equation}
where $t$ is the rank of the matrix $U_{l}^{tr}$. Now we may change
the representative $A$ to $A'=[U_1',U_2',\dots,U_r']^{tr}$ by the
action of the element
$$(\diag(1_{n_1},\ldots,P,\ldots,1_{n_{r-1}}), Q^{-1})$$
in $P_{I'}(k_F)\times\g{n_r}{k_F}$ such that $U_l'^{tr}$ is the matrix
(\ref{equation_level_6}).  If $t=n_l=n_r$ then condition $(2)$ is
satisfied.  Consider the maps $T_1:k_F^{n_l}\rightarrow k_F^{n_r}$ and
$T_2:k_F^{n_r}\rightarrow k_F^{n_l}$ given by
$$(a_1, a_2, \dots, a_{n_l})\mapsto (a_1, a_2, \dots, a_{n_l})U_l^{tr}$$
and 
$$(a_1,a_2,\dots, a_{n_r})\mapsto U_l^{tr}(a_1, a_2,\dots, a_{n_r})^{tr}$$
respectively.  If $t=n_l=n_r$ does not hold then either of $T_1$ or
$T_2$ has a non-trivial proper kernel (since $U_l\neq 0$). If $T_1$
has a non-trivial proper kernel then $M_{ll}$ preserves this kernel
and hence belongs to a proper parabolic subgroup of $\g{n_r}{k_F}$. If
$T_2$ has a non-trivial proper kernel then $B$ preserves this kernel
and hence belongs to a proper parabolic subgroup of
$\g{n_l}{k_F}$. Hence if $t=n_l=n_r$ does not hold then condition
$(1)$ is satisfied.
\end{proof}

The following lemma is due to Pa\v{s}k\={u}nas but we give a mild modification
for our applications (see \cite[Proposition
6.8]{Paskunas-uniqueness}).

\begin{lemma}\label{lemma_level_5_6} 
Let $G$ be a reductive
  algebraic group over $k_F$
  and $U$ be the unipotent radical of a proper
  parabolic subgroup of $G$. For any subgroup $H$ of $G$ such that
  $H\cap U=\{\id\}$ and $\xi$ any irreducible subrepresentation of
 $H$, there exists an irreducible non-cuspidal
  representation $\sigma$ such that $\xi$ occurs in $\res_H
  \sigma$. 
\end{lemma}

\begin{proof}
 Suppose the lemma is false then
$$\ind_{H}^{G}(\xi)\simeq\oplus_{k=1}^t\sigma_k$$
such that $\sigma_k$ is cuspidal representation for all $k\leq
t$. Since $U\cap H=\{\id\}$, using Mackey decomposition we deduce that,
$$\ho_{U}(\id,\ind_{H}^{G}(\gamma) )\neq 0.$$
Now by our assumption we have 
$\ho_{U}(\id, \sigma_k)\neq 0$ for some $k\leq t$ and hence a contradiction. 
\end{proof}

The following lemma is similar to Proposition
\ref{proposition_prelim_1}. The lemma is just a modified version of the
Proposition \ref{proposition_prelim_1} for our present use.

\begin{lemma}\label{lemma_level_7}
Let $\Gamma$ be a $K_{n}$-irreducible sub-representation of 
$$\ind_{P_{(n-n_r,n_r)}(m)}^{K_n}\{U_m(\tau_{I'})\boxtimes\tau_{r}\}.$$
If the irreducible sub-representations of $U_m(\tau_{I'})$ are
atypical for the component $s=[M_{I'}, \sigma_{I'}]$, then the
representation $\Gamma$ is atypical for the component
$s=[M_I, \sigma_I]$.
\end{lemma}

\begin{proof}
  Let $\rho$ be an irreducible sub-representation of $U_m(\tau_{I'})$.
  If $\rho$ is not typical then, there exists another Bernstein
  component $[M_{J}, \lambda_{J}]$ of $\g{n-n_r}{F}$ such that
$$[M_{I'}, \sigma_{I'}] \neq [M_{J}, \lambda_{J}]$$
and $\rho$ is contained in 
$$\res_{K_{n-n_r}}i_{P_{J}}^{G_{n-n_r}}(\lambda_{J})$$
where $J=(n_1',n_2',\dots,n'_{r'-1})$ and
$\lambda_J=\boxtimes_{i=1}^{r'-1}\lambda_i$. 
The representation  
$$\ind_{P_{(n-n_r,n_r)}(m)}^{K_n}\{\rho\boxtimes\tau_{r}\}$$
 is contained in 
\begin{equation}\label{equation_level_7}
\ind_{P_{(n-n_r,n_r)}\cap K_n}^{K_n}\{\rho\boxtimes\tau_{r}\}.
\end{equation}
The  representation (\ref{equation_level_7}) is contained in the representation 
$$\res_{K_n}i_{P_{(n-n_r,n_r)}}^{G_n}\{i_{P_{J}}^{G_{n-n_r}}(\lambda_{J})\boxtimes\sigma_r\}.$$

Since $[M_{I'}, \sigma_{I'}] \neq [M_{J}, \lambda_{J}]$ there exist an
inertial class $[G_p, \sigma]$ occurring in the multi-set
$$\{[G_{n_1},\sigma_1],[G_{n_2}, \sigma_2],\dots,[G_{n_{r-1}}, \sigma_{r-1}]\}$$ 
with a multiplicity not equal to its multiplicity in the multi-set 
$$\{[G_{n_1'},\lambda_1],[G_{n_2'},
\lambda_2],\dots,[G_{n_{r'-1}'}, \lambda_{r'-1}]\}.$$ 
Hence the classes $[M_I,\sigma_I]$ and $[M_J\times G_{n_r}, \lambda_J\boxtimes\sigma_r]$
represent two distinct Bernstein components for the group $G_n$.
\end{proof}

\section{Proof of the main theorem}\label{section_chapter1_1}
\begin{proof}[Proof of theorem \ref{theorem_level_1}]
  We prove the theorem by using induction on the positive integer $n$,
  the rank of $G_n$. The theorem is true for $n=1$ since
  $U_m(\tau_I)$ is zero. We assume that the theorem is true for all
  positive integers less than $n+1$. We will show the theorem for the
  positive integer $n+1$. Let $s=[M_I, \sigma_I]$ be a level-zero
  inertial class. We assume that the partition $I=(n_1,n_2,\dots,n_r)$
  of $n+1$ satisfies the hypothesis $n_i\leq n_j$ for all
  $1\leq i\leq j\leq r$.  If $r=1$ we have $U_m(\tau_I)=0$ and the
  theorem holds by default. We now assume that $r>1$ and let
  $I'=(n_1,n_2,\dots,n_{r-1})$.

We now break the proof into two cases. The first case is  $n_r=1$ and the second case is $n_r>1$.

\subsection*{The case where \texorpdfstring{$n_r=1$}{}} \label{subsection-level-1}
In this case $n_i=1$ for $1\leq i\leq r$ and $P_I=B_n$ where $B_n$ is
the Borel subgroup of $\text{GL}_n$. We denote by $T_n$ and $U_n$ the
maximal torus and the unipotent radical respectively.  We also use the
notation $B_n(m)$ for the subgroup $P_I(m)$ and $\chi_{I_n}$ for
$\tau_I$ since $I=(1,1,\dots,1)$ is a tuple of length $n$. The proof
is by induction on the integer $n$, the rank of $T_n$.  The statement
is immediate for $n=1$ and for $n=2$ we refer to
\cite[A.2.4]{Henniart-gl_2} for a proof. (We will require the proof for
later use and we will recall it at that stage.) So we prove the
theorem for $n\geq3$.  Suppose the theorem is true for some positive
integer $n\geq 2$. By the
definition of $U_m(\chi_{I_{n+1}})$ we have
$$\ind_{B_{n+1}(m)}^{K_{n+1}}(\chi_{I_{n+1}})\simeq
U_m(\chi_{I_{n+1}})\oplus 
\ind_{B_{n+1}(1)}^{K_{n+1}}(\chi_{I_{n+1}}).$$ 
We have the isomorphism 
$$\ind_{B_{n+1}(m)}^{K_{n+1}}(\chi_{I_{n+1}})\simeq
\ind_{P_{(n,1)}(m)}^{K_{n+1}}
\{\ind_{B_{n}(m)}^{K_n}(\chi_{I_n})\boxtimes\chi_{n+1})\}.$$
We also have the decomposition 
\begin{align*}
\ind_{P_{(n,1)}(m)}^{K_{n+1}}\{\ind_{B_{n}(m)}^{K_n}(\chi_{I_n})\boxtimes\chi_{n+1}\}\simeq& \\
\ind_{P_{(n,1)}(m)}^{K_{n+1}}\{U_m(\chi_{I_n})\boxtimes\chi_n\} \oplus
\ind_{P_{(n,1)}(m)}^{K_{n+1}}&\{\ind_{B_{n}(1)}^{K_n}(\chi_{I_n})\boxtimes\chi_{n+1}\}.
\end{align*}
By the inductive hypothesis and Lemma \ref{lemma_level_7} irreducible
sub-representations of
$$\ind_{P_{(n,1)}(m)}^{K_{n+1}}\{U_m(\chi_{I_n})\boxtimes\chi_{n+1}\}$$
are atypical representations. 
We now consider the irreducible factors of the representation
\begin{equation}\label{equation_case_1}
\ind_{P_{(n,1)}(m)}^{K_{n+1}}\{\ind_{B_n(1)}^{K_n}(\chi_{I_n})\boxtimes\chi_{n+1}\}.
\end{equation}
 We use induction on the integer $m$ to show that the representation 
\begin{align*}
\ind_{P_{(n,1)}(1)}^{K_{n+1}}\{\ind_{B_{n}(1)}^{K_n}(\chi_{I_n})\boxtimes\chi_{n+1}\}
\simeq \ind_{B_{n+1}(1)}^{K_{n+1}}(\chi_{I_{n+1}})
\end{align*}
has a complement say $U_{1,m}(\chi_{I{n+1}})$ in the representation
(\ref{equation_case_1}) whose irreducible sub-representations are all
atypical representations.  This shows that irreducible
sub-representations of $U_m(\chi_{I_{n+1}})$ are atypical.  To reduce
the notations we denote by $P(m)$ the subgroup $P_{(n,1)}(m)$. We note
that $P_{n,1}(m)=P_{(n,1)}(1,m)$ and now applying the decomposition
(\ref{equation_level_3}) to the parabolic subgroup $P_{(n,1)}$ we get
that
$$\ind_{P(m+1)}^{P(m)}(\id)=\id\oplus\ind_{Z(\eta)}^{P(m)}(U_{\eta})$$
where $\eta$
is any non-trivial character of the group $K_{n+1}(m)U_{n,1}(\integers{F})$ which is trivial on 
$K_{n+1}(m)U_{n,1}(\integers{F})\cap P(m+1)$ and $K_{n+1}(m)$ is the
principal congruence subgroup of level $m$ (in the present situation we just have one orbit
consisting of non-trivial characters.)  Let $\theta_{(n,1)}$ be the
isomorphism as in equation (\ref{equation_level_2.3}). We choose $\eta$ such
that $\theta_{(n,1)}(\eta)=[1,0\dots,0]$.

With the above choice for the character $\eta$ we have
\begin{align*}
&\ind_{P(m+1)}^{K_{n+1}}\{\ind_{B_n(1)}^{K_n}(\chi_{I_n})\boxtimes\chi_{n+1}\} 
 \\
&\simeq 
\ind_{P(m)}^{K_{n+1}}\{\ind_{B_n(1)}^{K_n}(\chi_{I_n})\boxtimes\chi_{n+1}\}
\oplus
 \ind_{Z(\eta)}^{K_{n+1}}\{U_{\eta}\otimes\res_{Z(\eta)\cap
  M_{(n,1)}}\{\ind_{B_n(1)}^{K_n}(\chi_{I_n})
\boxtimes\chi_{n+1}\}\}.
\end{align*}
Since the representation
$\ind_{B_n(1)}^{K_n}(\chi_{I_n})\boxtimes \chi_{n+1}$
is trivial on $K_n(1)\times K_1(1)$ we get that 
$$\res_{Z(\eta)\cap M_{(n,1)}}\{\ind_{B_n(1)}^{K_n}(\chi_{I_n})\boxtimes\chi_{n+1}\}$$
is isomorphic to the inflation of the representation 
$$\res_{\overline{Z(\eta)\cap M_{(n,1)}}}\{\ind_{B_n(k_F)}^{\g{n}{k_F}}(\chi_{I_n})\boxtimes\chi_{n+1}\}$$
where $\overline{Z(\eta)\cap M_{(n,1)}}$ is the mod-$\ideal{F}$
reduction of the group $Z(\eta)\cap M_{(n,1)}$. The group
$\overline{Z(\eta)\cap M_{(n,1)}}$ is contained in the following
subgroup
\begin{equation}
\label{equation_case_2}
\left\{\begin{pmatrix}A&B&0\\0&d&0\\0&0&d\end{pmatrix}| 
A\in \g{n-1}{k_F}, B\in M_{(n-1)\times1}(k_F)\ \text{and} \ d\in k_F^{\times}\right\}.
\end{equation}
 Let $\mir_k$ be the mirabolic group 
$$\left\{\begin{pmatrix}A&B\\0&1\end{pmatrix}|A\in \g{k-1}{k_F},
 B \in M_{(k-1)\times 1}(k_F) \right\}.$$
Now we have to understand the restriction 
$$\res_{P_{(n-1,1)}}\ind_{B_n(k_F)}^{\g{n}{k_F}}(\chi_{I_n}).$$
 We begin with understanding the restriction
$$\res_{\mir_{n-1}}\ind_{B_n(k_F)}^{\g{n}{k_F}}(\chi_{I_n}).$$
We use the theory of derivatives (originally for $G_n$ due to
Bernstein and Zelevinsky (see \cite{Bernstein_zelevinsky_1})) to
describe this restriction in a way sufficient for our application. We
refer to \cite[Chapter 3, \textsection 13]{Zelevinsky_finite} for
details of these constructions.

In the case of finite fields from Clifford theory one can define four
exact functors and we recall the formalism here. The precise definitions are
not required for our purpose except for one functor $\Psi^{+}$ which
will be recalled latter:
\begin{center}$\begin{tikzpicture}[node distance=3cm,auto]
\node(A){$\mathcal{M}(\mir_{k-1})$};
\node(B)[right of=A]{$\mathcal{M}(\mir_{k})$};
\node(C) [right of=B] {$\mathcal{M}(\g{k-1}{k_F})$};

\draw[->]([yshift=15pt] A) to node[above] {$\Phi^{+}$}(B);
\draw[->]([yshift=-15pt] B) to node[below] {$\Phi^{-}$}(A);
\draw[->]([yshift=10pt] B) to node[above] {$\Psi^{-}$}(C);
\draw[->]([yshift=-10pt]C) to node[below] {$\Psi^{+}$} (B);
\end{tikzpicture} $\end{center}
The key results we use from Zelevinsky are summarised below 
(see \cite[Chapter 3, \textsection 13]{Zelevinsky_finite}). 

\begin{theorem}[Zelevinsky]
  The functors $\Psi^{+}$
  and $\Phi^{-}$
  are left adjoint to $\Psi^{-}$
  and $\Phi^{+}$
  respectively. The compositions $\Phi^{-}\Phi^{+}$
  and $\Psi^{-}\Psi^{+}$
  are naturally equivalent to identity. Moreover
  $\Phi^{+}\Psi^{-}$ and $\Phi^{-}\Psi^{+}$ are zero. The diagram
$$0\rightarrow\Phi^{+}\Phi^{-}\rightarrow\id\rightarrow\Psi^{+}\Psi^{-}\rightarrow0$$ 
obtained from these properties is exact. 
\end{theorem}
Using this theorem and following Bernstein-Zelevinsky one can define a
filtration $Fil$
on a finite dimensional representation $\tau$
of $\mir_n$, for all $n>1$.  The filtration $Fil$ is given by
$$0\subset\tau_n\subset...\subset\tau_3\subset\tau_2\subset\tau_1=\tau$$ 
where $\tau_k=(\Phi^{+})^{k-1}(\Phi^{-})^{k-1}$
and
$\tau_k/\tau_{k+1}=(\Phi^{+})^{k-1}\Psi^{+}\Psi^{-}(\Phi^{-})^{k-1}(\tau)$
for all $k\geq
1$. The representation
$\tau^{(k)}:=\Psi^{-}(\Phi^{-})^{k-1}(\tau)$
for all $k\geq0$
of $\g{n-k}{k_F}$
is called the $k^{th}$-derivative
of $\tau$ and by convention $\tau^{(0)}:=\tau$.

Let $R_n$ be the Grothendieck group of  $\g{n}{k_F}$ for all
$n\geq1$ and set $R_0=\mathbb{Z}$. Zelevinsky defined a ring structure
on the group $R=\oplus_{n\geq0}R_n$ by setting parabolic induction as
the product rule. Recall that the ring $R$ has a $\mathbb{Z}$-linear
map 
$D$ defined by setting 
$D(\pi)=\sum_{k\geq0}(\pi|_{\mir_n})^{(k)}$ for all $\pi$ in $R_n$. It follows from
\cite[Chapter 3, \textsection 13]{Zelevinsky_finite} that
$D$ is an endomorphism of the ring $R$.  If $\pi$
is a supercuspidal representation of $\g{n}{k_F}$ then by
Gelfand-Kazhdan theory it follows that $\pi^{(n)}=1$, $\pi^{(0)}=\pi$
and all other derivatives are zero (see \cite[Chapter 3, \textsection
13]{Zelevinsky_finite}). Let $1_R\in R_0$ be the identity element of
$R$.

In our present situation we have 
$$D(\ind_{B_n(1)}^{K_n}(\chi_{I_n}))=\prod_{i=1}^{n}D(\chi_i)=\prod_{i=1}^{n}(\chi_i+1_R).$$
Let $X_{n-k}$ be the term of degree $(n-k)$ in the expansion of the
above product. (It is an actual  representation of
$\g{n-k}{k_F}$, since the
coefficients of the above expansion are positive.) Then we have
$$\res_{\mir_{n-1}}\ind_{B_n(k_F)}^{\g{n}{k_F}}(\chi_{I_n})\simeq 
\bigoplus_{k\geq1}^n(\Phi^+)^{k-1}\Psi^{+}(X_{n-k}).$$
Observe that $P_{(n-1,1)}=\mir_{(n-1)}k_F^{\times}$ (here
$k_F^{\times}$ is the centre of $\g{n}{k_F}$) and
$\mir_{(n-1)}\cap k_F^{\times}=\id$. The representation
$$\rho:=(\Phi^{+})^{k-1}\Psi^{+}(X_{n-k})$$
extends to a representation of $P_{(n-1,1)}$ by setting
$\rho(a)=\chi(a)$ for all $a\in k_F^{\times}$ where $\chi$ is the
central character of the representation
 $$\ind_{B_{n}(k_F)}^{\g{n}{k_F}}(\boxtimes_{i=1}^{n}\chi_i).$$
 Since the central character will play some role, we denote the
 extended representation by
 $$ext\{(\Phi^{+})^{k-1}\Psi^{+}(X_{n-k})\}.$$ 
 By inflation and (later restriction) we  extend the 
$P_{(n-1,1)}(k_F)\times k_F^{\times}$-representation
$$ext\{(\Phi^{+})^{k-1}\Psi^{+}(X_{n-k})\}\boxtimes\chi_{n+1}$$ 
to a representation of $Z(\eta)\cap M_{(n,1)}$. We continue to use the notation 
$$ext\{(\Phi^{+})^{k-1}\Psi^{+}(X_{n-k})\}\boxtimes\chi_{n+1}$$
considered as a representation of $Z(\eta)\cap M_{(n,1)}$.  We now
have
\begin{align*}\ind_{P(m+1)}^{K_{n+1}}(\chi_{I_n})\simeq &\ind_{P(m)}^{K_{n+1}}(\chi_{I_n})\oplus\\
                                                                      &
                                                                        \bigoplus_{k\geq1}^n \ind_{Z(\eta)}^{K_{n+1}}\{ext\{(\Phi^{+})^{k-1}\Psi^{+}(X_{n-k})\}\boxtimes\chi_{n+1}\}.\end{align*}

We will show that any irreducible sub-representation of 
$$\ind_{Z(\eta)}^{K_{n+1}}(ext\{(\Phi^{+})^{k-1}\Psi^{+}(X_{n-k})\}\boxtimes\chi_{n+1})\}$$
is atypical for the component $[T_n, \chi_{I_n}]$.

We first consider the case when $k\geq 2$.  The representation
$X_{n-k}$ is a direct sum of the representations:
$$\ind_{B_{n-k}(k_F)}^{\g{n-k}{k_F}}(\chi_{i_1}\boxtimes\chi_{i_2}
\boxtimes\dots\boxtimes\chi_{i_{n-k}}).$$
The above term also occurs in the expansion 
$$\prod_{j=1}^{n-k}(1_R+\chi_{i_j})(1_R+\lambda)$$
where $\lambda$ is a cuspidal representation of $\g{k}{k_F}$.  To
shorten the notation we use the symbol $\times$ for the multiplication
in the ring $R$. We get that the representation
$$(\Phi^{+})^{k-1}\Psi^{+}(\chi_{i_1}\times\chi_{i_2}\times\ldots\times\chi_{i_{n-k}})$$ 
occurs in the representation 
$$\res_{\mir_{n-1}}(\chi_{i_1}\times\chi_{i_2}\times\ldots\times\chi_{i_{n-k}}\times\lambda).$$

Note that the mod-$\ideal{F}$ reduction of the group
$Z(\eta)\cap M_{(n,1)}$ is contained in the subgroup of the form
\eqref{equation_case_2} and recall that the $(n-1)^{th}$ diagonal
entry of any element in \eqref{equation_case_2}  is
the same as its $n^{th}$ diagonal entry. Let $\eta_1$ and $\eta_2$ be the central characters of
$\chi_1\times\chi_2\times\ldots\times\chi_n$ and
$(\chi_{i_1}\times\chi_{i_2}\times\ldots\times\chi_{i_{n-k}})\times\lambda$
respectively.
Let $\chi_{n+1}'$ be the character $\chi_{n+1}\eta_2^{-1}\eta_1$ and now 
the representation
$$\res_{Z(\eta)\cap M_{(n,1)}}\{ext((\Phi^{+})^{k-1}\Psi^{+}
(\chi_{i_1}\times\chi_{i_2}\times\ldots\times\chi_{i_{n-k}}))\}\boxtimes\chi_{n+1}$$
occurs in the representation 
$$\res_{Z(\eta)\cap M_{(n,1)}}
(\chi_{i_1}\times\chi_{i_2}\times\ldots\times\chi_{i_{n-k}})\times\lambda)
\boxtimes\chi_{n+1}'.$$
Hence  an irreducible sub-representation of 
\begin{equation}\label{equation_case_2.5}
\ind_{Z(\eta)}^{K_{n+1}}\{(ext\{(\Phi^{+})^{k-1}
\Psi^{+}(X_{n-k})\}\boxtimes\chi_{n+1})\otimes U_{\eta}\}
\end{equation}
occurs as a sub-representation of
\begin{equation}\label{equation_case_3}
\ind_{Z(\eta)}^{K_{n+1}}
\{\{(\chi_{i_1}\boxtimes\chi_{i_2}\boxtimes\dots\boxtimes\chi_{i_{n-k}}
\boxtimes\lambda\boxtimes\chi_{n+1}')\}\otimes U_{\eta}\}.
\end{equation}
The above representation occurs as a sub-representation of 
\begin{equation}\label{equation_case_4}
\ind_{P_{(1,1,\dots,k,1)}\cap
  K_{n+1}}^{K_{n+1}}\{\chi_{i_1}\boxtimes\chi_{i_2}
\boxtimes\dots\boxtimes\chi_{i_{n-k}}\boxtimes\lambda\boxtimes\chi_{n+1}'\}.
\end{equation}
Hence the sub-representation of (\ref{equation_case_2.5}) are not typical representations. 

Now we are left with the term 
\begin{equation}\label{equation_case_5}
\ind_{Z(\eta)}^{K_{n+1}}\{(ext\{\Psi^{+}(X_{n-1})\}\boxtimes\chi_{n+1})\otimes U_{\eta}\}.
\end{equation}

If we repeat the same strategy as for $k\geq2$ then
$\lambda$ is one-dimensional so the representations
$(\ref{equation_case_4})$ and $\chi_1\times\chi_2\times\ldots\chi_{n+1}$ may not have
distinct inertial support. In order to tackle the terms of the above
representation we use a different technique. We now recall the
definition of the representation $U_{\eta}$, the functor $\Psi^{+}$
and some facts due to Casselman regarding the restriction of an
irreducible smooth representation to the maximal compact subgroup
$\g{2}{\integers{F}}$.

The representation $U_{\eta}$ is a character on the group
$Z(\eta)$. From $(\ref{equation_case_2})$ any element of the group
$Z(\eta)$ is of the form
\begin{equation}\label{equation_case_5.5}
\begin{pmatrix}A&B&X'\\\varpi_FC&d&y\\\varpi_F^mX&\varpi_F^my'&e\end{pmatrix} 
\end{equation}
where $A\in \g{n-1}{\integers{F}}$; $(X'), X^{tr},B,C^{tr}\in
M_{(n-1)\times1}(\integers{F})$; 
$e,d\in \mathcal{O}_F^{\times}$; $y, y'\in \integers{F}$ and $d\equiv e \pmod{\ideal{F}}$.
The character  $U_{\eta}$ is given by 
$$\begin{pmatrix}A&B&X'\\\varpi_FC&d&y\\\varpi_F^mX&\varpi_F^my'&e\end{pmatrix}
\mapsto \eta(\varpi_F^my').$$
The functor 
$$\Psi^{+}:\mathcal{M}(\g{k-1}{k_F})\rightarrow \mathcal{M}(\mir_k)$$ 
is the inflation functor via the quotient map of $\mir_k$ modulo the
unipotent radical of 
$\mir_k$.

Let $(\pi,V_{\pi})$ be an irreducible smooth representation of
$\g{2}{F}$. We denote by $c(\pi)$ and $\varpi_{\pi}$ the conductor
(see \cite[Theorem 1]{casselman_atkin}) and
central character of the representation $\pi$ respectively.  Let
$V^{N}$ be the space of all vectors fixed by the principal congruence
subgroup of level $N$ for all $N\geq 1$. For all $i> c(\varpi_{\pi})$
we define the representation $U_i(\varpi_{\pi})$ as the complement of the
representation $\ind_{B_2(i-1)}^{\g{2}{\integers{F}}}(\varpi_{\pi})$ in
$\ind_{B_2(i)}^{\g{2}{\integers{F}}}(\varpi_{\pi})$. For $i=c(\varpi_{\pi})$
we set
$$U_i(\varpi_{\pi})=\ind_{B_2(i)}^{\g{2}{\integers{F}}}(\varpi_{\pi}\boxtimes\id).$$
It follows from \cite[Proposition 1]{Casselmanres} that the
representation $U_i(\varpi_{\pi})$ is an irreducible representation of
$\g{2}{\integers{F}}$. From the result \cite[Proposition
2]{Casselmanres} we get that $c(\pi)\geq c(\varpi_{\pi})$. By
\cite[Theorem 1]{Casselmanres} we have
\begin{equation}\label{equation_level_casselman}
\res_{\g{2}{\integers{F}}}V_{\pi}=V^{(c(\pi)-1)}\oplus\bigoplus_{i\geq c(\pi)} U_{i}(\varpi_{\pi}).
\end{equation}
We now describe the representation $U_i(\varpi_{\pi})$ in our
language. Let $\eta$ be a non-trivial character of the group
$K_2(m)U_{(1,1)}(\integers{F})$
trivial modulo 
$$K_2(m)U_{(1,1)}(\integers{F})\cap B_2(m+1).$$
Let $Z(\eta)$ be the $B_2(m)$ stabilizer of $\eta$.  Any element of
the group $Z(\eta)$ is of the form
$$\begin{pmatrix}a&b\\c&d\end{pmatrix}$$
where $a,d\in \mathcal{O}_F^{\times}$; $b\in \integers{F}$,
$c\in \mathfrak{P}_F^m$ and $d\equiv a$ modulo $\ideal{F}$.  We define
a character $U_{\eta}$ by setting
$$\begin{pmatrix}a&b\\c&d\end{pmatrix} \mapsto \eta(c).$$
We then have 
$$U_m(\varpi)\simeq \ind_{Z(\eta)}^{K_2}(U_{\eta}\otimes(\varpi\boxtimes\id)).$$

Now let us resume the proof in the general case $n> 2$ the
representation
$$\ind_{Z(\eta)}^{K_{n+1}}\{(ext\{\Psi^{+}(X_{n-1})\}\boxtimes\chi_{n+1})\otimes U_{\eta}\}$$
is contained in the representation 
\begin{equation}\label{equation_case_6}
\ind_{P_{(n-1,2)}(m)}^{K_{n+1}}(X_{n-1}\boxtimes U_m(\chi))
\end{equation}
where $\chi$ is given by $\prod_{i=1}^n\chi_i$ of
$\mathcal{O}_F^{\times}$. This representation, by the theorem of
Casselman (see the decomposition \eqref{equation_level_casselman}) is
contained in the representation
$$\ind_{P_{(n-1,2)}\cap K_{n+1}}^{K_{n+1}}(X_{n-1}'\boxtimes \sigma)$$
where $\sigma$ is a supercuspidal representation of level-zero with
central character $\chi$ (see the remark below for the existence) and
$X_{n-1}'$ is the $(n-1)^{\rm th}$ derivative of the representation
$$i_{B_n}^{G_n}(\chi_{I_n}).$$
Hence irreducible sub-representations of \eqref{equation_case_5} are
atypical. This completes the proof that irreducible
sub-representations of
$$ \ind_{Z(\eta)}^{K_{n+1}}\{U_{\eta}\otimes\res_{Z(\eta)\cap M_{(n,1)}}\{\ind_{B_n(1)}^{K_n}(\chi_{I_n})\boxtimes\chi_{n+1}\}\}$$
are atypical. From the  decomposition  
\begin{align*}
&\ind_{P(m+1)}^{K_{n+1}}\{\ind_{B_n(1)}^{K_n}(\chi_{I_n})\boxtimes\chi_{n+1}\} \\
&\simeq 
\ind_{P(m)}^{K_{n+1}}\{\ind_{B_n(1)}^{K_n}(\chi_{I_n})\boxtimes\chi_{n+1}\}\\
& \hspace{10mm}\oplus
 \ind_{Z(\eta)}^{K_{n+1}}\{U_{\eta}\otimes\res_{Z(\eta)\cap M_{(n,1)}}\{\ind_{B_n(1)}^{K_n}(\chi_{I_n})\boxtimes\chi_{n+1}\}\}.
\end{align*}
we get the theorem for the case where $n_r=1$.

\begin{remark}
  The existence of the cuspidal representation of $\g{2}{k_F}$ with a
  given central character can be deduced from the explicit formula for
  such representations, we refer to \cite[Theorem section
  6.4]{Yellowbook}. To be precise we begin with a quadratic extension
  $k$ of $k_F$ and $\theta$ a character of $k^{\times}$ such that
  $\theta^q\neq \theta$ where $q=\#k_F$. These characters are called
  regular characters and for any regular character one can define a
  supercuspidal representation $\pi_{\theta}$ and conversely all
  supercuspidal representations are of the form $\pi_{\theta}$ for
  some regular character $\theta$. The central character of
  $\pi_{\theta}$ is given by $\res_{k_F^{\times}}(\theta)$.  Now to
  get a supercuspidal representation with a central character $\chi$
  we begin with a character $\chi$ on $k_F^{\times}$, there are
  $\#k_F+1$ possible extensions to $k^{\times}$. The set of characters
  $\theta$ such that $\theta^q=\theta$ has cardinality $\#k_F-1$. Hence there
  exists at least one supercuspidal representation with a given
  central character $\chi$.  This shows that irreducible
  sub-representations of $(\ref{equation_case_6})$ are not typical and
  this completes the proof of the theorem in this case.
\end{remark}

\subsection*{The case where \texorpdfstring{$n_r>1$}{}}
By transitivity of induction we have
$$\ind_{P_I(m)}^{P_I(1)}(\tau_I)\simeq \ind_{P_I(1,m)}^{P_I(1)}\{\ind^{P_I(1,m)}_{P_I(m)}(\tau_I)\}.$$
We note that $P_I(1,m)\cap U_{(n-n_r+1,n_r)}$ is equal to $P_I(m)\cap U_{(n-n_r+1,n_r)}$ and 
$P_I(1,m)\cap \bar{U}_{(n-n_r+1,n_r)}$ is equal to $P_I(m)\cap \bar{U}_{(n-n_r+1,n_r)}$ hence Lemma
\ref{lemma_prelim_3} gives the isomorphism 
$$\ind_{P_I(1,m)}^{P_I(1)}\{\ind^{P_I(1,m)}_{P_I(m)}(\tau_I)\}\simeq
\ind_{P_I(1,m)}^{P_I(1)}
\{(\ind_{P_{I'}(m)}^{P_{I'}(1)}(\tau_{I'}))\boxtimes\tau_r)\}.$$
Splitting the representation $\ind_{P_{I'}(m)}^{P_{I'}(1)}(\tau_{I'})$
as $\tau_{I'}\oplus U_m^0(\tau_{I'})$ we get that
$$
\ind_{P_I(1,m)}^{P_I(1)}\{(\ind_{P_{I'}(m)}^{P_{I'}(1)}(\tau_{I'}))\boxtimes\tau_r)\}\simeq 
\ind_{P_I(1,m)}^{P_I(1)}\{U^0_m(\tau_{I'})\boxtimes\tau_r\}\oplus \ind_{P_I(1,m)}^{P_I(1)}(\tau_I).$$
From Frobenius reciprocity the representation $\tau_I$ occurs in
$\ind_{P_I(1,m)}^{P_I(1)}(\tau_I)$ with multiplicity one. Let
$U_{(1,m)}^0(\tau_I)$ be the complement of $\tau_I$ in
$\ind_{P_I(1,m)}^{P_I(1)}(\tau_I)$.  With this we conclude that
$$\ind_{P_I(m)}^{P_I(1)}(\tau_I)\simeq
\ind_{P_I(1,m)}^{P_I(1)}\{U^0_m(\tau_{I'})\boxtimes\tau_r\}
\oplus U_{(1,m)}^0(\tau_I)\oplus \tau_I.$$
By definition
$U_{m}(\tau_I)=\ind_{P_I(1)}^{K_{n+1}}(U_m^0(\tau_I))$
which shows that
$$U_{m}(\tau_I)\simeq
\ind_{P_I(1,m)}^{K_{n+1}}\{U^0_m(\tau_{I'})\boxtimes\tau_r\}
\oplus \ind_{P_I(1)}^{K_{n+1}}(U_{(1,m)}^0(\tau_I)).$$
We observe that 
$$P_I(1,m)\cap U_{(n-n_r+1,n_r)}=P_{(n-n_r+1,n_r)}(m)\cap
U_{(n-n_r+1,n_r)}$$
 and 
$$P_I(1,m)\cap \bar{U}_{(n-n_r+1,n_r)}=P_{(n-n_r+1,n_r)}(m)\cap \bar{U}_{(n-n_r+1,n_r)}$$ 
hence Lemma \ref{lemma_prelim_3} applied to the groups $J_2=P_I(1,m)$
and $J_1=P_{(n-n_r+1,n_r)}(m)$ and
$\lambda=U^0_m(\tau_{I'})\boxtimes\tau_r$ gives us the isomorphism
$$\ind_{P_I(1,m)}^{K_{n+1}}\{U^0_m(\tau_{I'})\boxtimes\tau_r\}\simeq 
\ind_{P_{(n-n_r+1,n_r)}(m)}^{K_{n+1}}\{U_m(\tau_{I'})\boxtimes\tau_r\}.$$
With this we are in a place to use the induction hypothesis through
the isomorphism
\begin{equation}\label{equation_proof_1}
U_{m}(\tau_I)\simeq
\ind_{P_{(n-n_r+1,n_r)}(m)}^{K_{n+1}}\{U_m(\tau_{I'})\boxtimes\tau_r\}
\oplus \ind_{P_I(1)}^{K_{n+1}}(U_{(1,m)}^0(\tau_I)).
\end{equation}

By induction hypothesis $\g{n-n_r+1}{\integers{F}}$-irreducible
sub-representations of $U_m(\tau_{I'})$ are atypical for the component
$[M_{I'}, \sigma_{I'}]$. Now Lemma \ref{lemma_level_7} and the
equation \eqref{equation_proof_1} reduce the proof of the theorem to
showing that irreducible sub-representations of
$\ind_{P_I(1)}^{K_{n+1}}(U_{(1,m)}^0(\tau_I))$ are
atypical representations.

\begin{proposition}
 The irreducible sub-representations of 
 $$\ind_{P_I(1)}^{K_{n+1}}(U_{(1,m)}^0(\tau_I))$$
  are atypical for $m\geq1$.
\end{proposition}
\begin{proof}

We observe that
$$\ind_{P_I(1,m+1)}^{P_I(1)}(\tau_I)\simeq \ind_{P_I(1,m)}^{P_I(1)}\{\ind_{P_I(1,m+1)}^{P_I(1,m)}(\tau_I)\}$$
and the decomposition (\ref{equation_level_3}) gives us the isomorphism 
$$\ind_{P_I(1,m+1)}^{P_I(1)}(\tau_I)= \ind_{P_I(1,m)}^{P_I(1)}(\tau_I)
\oplus 
\bigoplus_{\eta_k\neq \id}\ind_{P_I(1,m)}^{P_I(1)}\{(\ind_{Z(\eta_k)}^{P_I(1,m)}(U_{\eta_k})\otimes\tau_I\}$$
which gives the equality 
$$U_{(1,m+1)}^0(\tau_I)=U_{(1,m)}^0(\tau_I) \oplus 
\bigoplus_{\eta_k\neq \id}\ind_{P_I(1,m)}^{P_I(1)}\{(\ind_{Z(\eta_k)}^{P_I(1,m)}(U_{\eta_k}))\otimes\tau_I\}.$$

If we show that the irreducible sub-representations of
$$\ind_{P_I(1,m)}^{K_{n+1}}\{(\ind_{Z(\eta_k)}^{P_I(1,m)}(U_{\eta_k}))\otimes\tau_I\}$$
(for $\eta_k\neq \id$) are atypical for $[M_I, \sigma_I]$ then
induction on the positive integer $m$ completes the proof of the
proposition in this case.  To begin with we note that
\begin{align*}
\ind_{P_I(1,m)}^{K_{n+1}}\{\ind_{Z(\eta_k)}^{P_I(1,m)}(U_{\eta_k})\otimes\tau_I\}
  \simeq  \ind_{P_I(1,m)}^{K_{n+1}}
\{\ind_{Z(\eta_k)}^{P_I(1,m)}(U_{\eta_k}\otimes\res_{Z(\eta_k)\cap M_I}\tau_I)\}.
\end{align*}

The representation $\tau_I$ is trivial on $M_I\cap K_{n+1}(1)$. Hence
$\res_{Z(\eta_k)\cap M_I}\tau_I$ is isomorphic to the inflation of the
representation $\res_{\overline{Z(\eta_k)\cap M_I}}\tau_I$ where
$\overline{Z(\eta_k)\cap M_I}$ is mod-$\ideal{F}$ reduction of
${Z(\eta_k)\cap M_I}$.  Let $A=\theta_I(\eta_k)$ where $\theta_I$ is
the map defined in equation \eqref{equation_level_2.3}. The
mod-$\ideal{F}$ reduction $\overline{Z(\eta_k)\cap M_I}$ is contained
in $Z_{P_{I'}(k_F)\times\g{n_r}{k_F}}(A)$. If $\eta_k$ is a nontrivial
character then $A\neq 0$.  Applying Lemma \ref{lemma_level_4} to the
parabolic corresponding to the partition $I=(n_1, n_2, \dots , n_r)$
and the assumption that $n_r>1$ shows that
$$Z_{P_{I'}(k_F)\times\g{n_r}{k_F}}(A) \cap M_I(k_F)$$
has trivial intersection with a unipotent subgroup $U$ of
$M_I(k_F)$. If $A$ satisfies condition $(1)$ of the Lemma
\ref{lemma_level_4} then let $U_j$ be the unipotent radical of the
opposite parabolic subgroup containing the image of $p_j$ (see Lemma
\ref{lemma_level_4}) and if $A$ satisfies condition $(2)$ of
\ref{lemma_level_4} then let $U_r$ be any unipotent radical of a
parabolic subgroup of $\g{n_r}{k_F}$. Now $U$ can be chosen to be the
group
$$\{(u_1, u_2, \dots, u_r) \ \in M_I(k_F) \ | u_j \in U_j \
\text{and} \ u_k=1_k \ \forall \ k\neq j\}$$
if $A$ satisfies condition $(1)$ in Lemma \ref{lemma_level_4} and $U$ to
be
$$\{(u_1, u_2, \dots, u_r) \ \in M_I(k_F) \ | u_r \in U_r \
\text{and} \ u_k=1_k\ \forall \ k\neq r\}$$
if $A$ satisfies the condition $(2)$ in Lemma \ref{lemma_level_4}. 

Now applying Lemma \ref{lemma_level_5_6} to the reductive group
$G=M_I$ and the representation $\tau=\tau_I$, we get that every
irreducible subrepresentation $\Gamma$ of

$$\ind_{P_I(1,m)}^{K_{n+1}}\{\ind_{Z(\eta_k)}^{P_I(1,m)}(U_{\eta_k}\otimes\res_{Z(\eta_k)\cap
  M_I}\tau_I)\}$$
occurs as a subrepresentation of some representation of the form 
\begin{equation}\label{equation_last_1}
\ind_{P_I(1,m)}^{K_{n+1}}\{\ind_{Z(\eta_k)}^{P_I(1,m)}(U_{\eta_k}\otimes\res_{Z(\eta_k)\cap
  M_I}\tau_I')\},
\end{equation}
with $\tau_I'$  some non-cuspidal representation of $M_I(k_F)$. The
representation \ref{equation_last_1} occurs as a subrepresentation of 
$$\ind_{P_I\cap K_{n+1}}^{K_{n+1}}\tau_I',$$
which occurs as a subrepresentation of 
$$\res_{K_{n+1}}i_{P_I}^{G_{n+1}}\sigma_I'$$
where $\sigma_I'$ is a non-cuspidal representation such that
$\sigma_I'^{K_{n+1}(1)\cap M_I}\simeq \tau_I'$. 
Hence the representation $\Gamma$ is not a typical representation for
the component $[M_I, \sigma_I]$. 

 This completes the proof of the proposition and also the proof of the theorem.
\end{proof}

\end{proof}

\bibliography{./biblio}
\bibliographystyle{amsalpha}

\noindent Santosh Nadimpalli,\\
School of Mathematics, Tata Institute of Fundamental Research,
Mumbai, 400005.\\
\texttt{nvrnsantosh@gmail.com}, \texttt{nsantosh@math.tifr.res.in}

\end{document}